\documentclass{article}%
\usepackage{amssymb}
\usepackage{graphicx}
\usepackage{amsmath}%
\usepackage{amsfonts}
\providecommand{\U}[1]{\protect\rule{.1in}{.1in}}
\newtheorem{theorem}{Theorem}

\newtheorem{corollary}[theorem]{Corollary}

\newtheorem{definition}[theorem]{Definition}

\newtheorem{lemma}[theorem]{Lemma}
\newtheorem{notation}[theorem]{Notation}

\newtheorem{proposition}[theorem]{Proposition}
\newtheorem{remark}[theorem]{Remark}

\newenvironment{proof}[1][Proof]{\textbf{#1.} }{\ \rule{0.5em}{0.5em}}
\begin{document}

\title{Axiomatic Differential Geometry I-1\\-Towards Model Categories of Differential Geometry-}
\author{Hirokazu Nishimura\\Institute of Mathematics\\University of Tsukuba\\Tsukuba, Ibaraki, 305-8571, JAPAN}
\maketitle

\begin{abstract}
In this paper we give an axiomatization of differential geometry comparable to
model categories for homotopy theory. Weil functors play a predominant role.

\end{abstract}

\section{Introduction}

It is well known that the category of topological spaces and continuous
mappins is by no means cartesian closed, which has harassed algebraic
topologists. In 1967 Steenrod \cite{st} popularized the idea of \underline
{convenient category} by announcing that the category of compactly generated
spaces and continuous mappings renders a good setting for algebraic topology.
The advertised category is cartesian closed, complete and cocomplete, and
contains all CW complexes. In the same year, Quillen \cite{qu}\ finally
succeeded in axiomatizing \underline{homotopy theory}, which is now known as
\underline{model categories}.

Turning to \underline{differential geometry}, more than a few geometers have
tried to give a convenient category for differential geometry, say,
\cite{ch1}, \cite{ch2}, \cite{ch3}, \cite{fr1}, \cite{fr2}, \cite{fr3},
\cite{kri}, \cite{mos}, \cite{si}, \cite{sm} and \cite{sou}. Some acute
mathematicians compared these proposed convenient categories, say, \cite{ba1}
and \cite{s1}. Now what is completely missing is an axiomatization of
differential geometry comparable to model categories for homotopy theory. We
hastily home in on one so as to fill in the rift, at least as far as
infinitesimal aspects of differential geometry are concerned.

\underline{Weil algebras} were introduced by Weil himself \cite{we}. They were
intended for the algebraic realization of fabulous \underline{nilpotent
infinitesimals}. It is presumably synthetic differential geometers who have
used Weil algebras systematically in differential geometry for the first time.
They prefer to enjoy tangled relations among various Weil algebras. In
particular, they have reached the crucial notion of \underline{microlinearity}%
. To synthetic differential geometers, Weil functors are merely the
exponentiation by infinitesimal objects corresponding to Weil algebras, while,
to orthodox differential geometers, they are a natural generalization of the
tangent bundle functor so that they can be defined without any reference to
legendary infinitesimal objects. Roughly speaking, our axiomatization of
differential geometry is a convenient category endowed with Weil functors.
Generally speaking, any proposed convenient category is so broad as to contain
spaces which are not necessarily amenable to the study by methods of
differential geometry. From our standpoint, the notion of \underline{manifold}
is a flawed concept, or politely saying, a transitory concept to be replaced
by another more appropriate one, just as Riemann integrals were to be replaced
by Lebesgue integrals. It is the notion of microlinearity that enables us to
delineate the class of spaces adequate for the study of differential geometry.
It gives us a great pleasure to see that the full subcategory of the
convenient category consisting of all spaces susceptible of
differential-geometric investigation is cartesian closed, whatever the
convenient category may be. We will discuss our axiomatization in \S \ref{s3}.

In orthodx differential geometry, just as smooth manifolds are spaces which
are locally Euclidean (namely, locally diffeomorphic to some open subsets of
$\mathbf{R}^{m}$), \underline{fibered manifolds} are locally the canonical
projections $\mathbf{R}^{m+n}\rightarrow\mathbf{R}^{m}$. As Mangiarotti and
Modugno \cite{mm} have stressed, a large portion of differential geometry (at
least up to connections and their related concepts) could be developed upon
fibered manifolds. We should say that the othodox notion of fibered manifold
is slightly distorted, simply because the map is required to be a submersion
so as to make every emerging entity amenable to the realm of manifolds. From
our standpoint, the story goes as follows. Given a convenient category
provided with Weil functors, its arrow category is also naturally endowed with
derived Weil functors. Our notion of fiber bundle is simple enough. It is
microlinear objects in the arrow category. This point will be discussed in
detail in \S \ref{s4}. In \S \ref{s5}, we will discuss \underline{vertical
Weil functors}.

\section{Preliminaries}

\subsection{Category Theory}

There are many good textbooks on category theory. By way of example, \cite{ma}
and \cite{sch} are recommendable classics. Therefore it would be absurd to try
to explain category theory from scratch. However we must fix our own notation
and terminology in this arena. A category $\mathcal{K}$\ is called
\underline{left exact} if it has finite limits. A functor between left exact
categories is called \underline{left exact} if it preserves finite limits. A
\underline{diagram} in a category $\mathcal{K}$\ is a functor $\mathcal{D}%
$\ from a category $\Lambda$\ to the category $\mathcal{K}$. Its limit in
$\mathcal{K}$ is usually denoted by $\mathrm{Lim}_{\lambda\in\Lambda
}\mathrm{\ }\mathcal{D}_{\lambda}$. Given a natural transformation
$\rho:\mathcal{J}\overset{\cdot}{\rightarrow}\mathcal{K}$ between two functors
$\mathcal{F},\mathcal{G}:\mathcal{J}\rightarrow\mathcal{K}$ and an object $X$
in $\mathcal{J}$, the morphism $\mathcal{F}(X)\rightarrow\mathcal{G}\left(
X\right)  $ induced by the natural transformation $\rho$\ is denoted by
$\rho(X)$ or $\rho_{X}$.Given a category $\mathcal{K}$, its arrow category is
usually denoted by $\overrightarrow{\mathcal{K}}$ in preference to
$\mathcal{K}^{\rightarrow}$.

\subsection{Weil Algebras}

Let $k$\ be a commutative ring. The category of Weil algebras over $k$\ (also
called Weil $k$-algebras) is denoted by $\mathbf{Weil}_{k}$. It is well known
that the category $\mathbf{Weil}_{k}$\ is left exact. The terminal object in
$\mathbf{Weil}_{k}$\ is $k$ itself, and, given an object $W$\ in
$\mathbf{Weil}_{k}$, the unique morphism $W\rightarrow k$ in $\mathbf{Weil}%
_{k}$ is denoted by $\underline{\tau}_{W}$. Since any object $W$\ in
$\mathbf{Weil}_{k}$ is a $k$-algebra, there is a canonical morphism
$k\rightarrow W$, which we denote by $\underline{\iota}_{W}$. Given two
objects $W_{1}$ and $W_{2}$, we denote their tensor algebra by $W_{1}%
\otimes_{k}W_{2}$. For a good treatise on Weil algebras, the reader is
referred to \S \ 1.16 of \cite{ko}. Given a left exact category $\mathcal{K}%
$\ and a $k$-algebra object $\mathbb{R}$\ in $\mathcal{K}$, there is a
canonical functor $\mathbb{R}\underline{\otimes}\cdot$\ (denoted by
$\mathbb{R\otimes\cdot} $\ in \cite{ko}) from the category $\mathbf{Weil}_{k}$
to the category of $k$-algebra objects and their homomorphisms in
$\mathcal{K}$.

\section{\label{s3}Axiomatics}

\begin{definition}
A \underline{DG-category} (DG stands for Differential Geometry) is a quadruple
$\left(  \mathcal{K},\mathbb{R},\mathbf{T},\alpha\right)  $, where

\begin{enumerate}
\item $\mathcal{K}$ is a category which is left exact and cartesian closed.

\item $\mathbb{R}$ is a commutative $k$-algebra object in $\mathcal{K}$.

\item Given a Weil $k$-algebra $W$, $\mathbf{T}^{W}:\mathcal{K}\rightarrow
\mathcal{K}$ is a left exact functor for any Weil $k$-algebra $W$ subject to
the condition that $\mathbf{T}^{k}:\mathcal{K}\rightarrow\mathcal{K}$ is the
identity functor, while we have
\begin{equation}
\mathbf{T}^{W_{2}}\circ\mathbf{T}^{W_{1}}=\mathbf{T}^{W_{1}\otimes_{k}W_{2}%
}\label{3.4}%
\end{equation}
for any Weil $k$-algebras $W_{1}$ and $W_{2}$.

\item Given a Weil $k$-algebra $W$, we have
\[
\mathbf{T}^{W}\mathbb{R}=\mathbb{R}\underline{\otimes}W
\]

\item $\alpha_{\varphi}:\mathbf{T}^{W_{1}}\overset{\cdot}{\rightarrow
}\mathbf{T}^{W_{2}}$ is a natural transformation for any morphism
$\varphi:W_{1}\rightarrow W_{2}$ in the category $\mathbf{Weil}_{k}$ such that
we have
\[
\alpha_{\psi}\cdot\alpha_{\varphi}=\alpha_{\psi\circ\varphi}%
\]
for any morphisms $\varphi:W_{1}\rightarrow W_{2}$ and $\psi:W_{2}\rightarrow
W_{3}$ in the category $\mathbf{Weil}_{k}$, while we have
\[
\alpha_{\mathrm{id}_{W}}=\mathrm{id}_{\mathbf{T}^{W}}%
\]
for any identity morphism $\mathrm{id}_{W}:W\rightarrow W$ in the category
$\mathbf{Weil}_{k}$.

\item Given a morphism $\varphi:W_{1}\rightarrow W_{2}$ in the category
$\mathbf{Weil}_{k}$, we have
\[
\alpha_{\varphi}\left(  \mathbb{R}\right)  =\mathbb{R}\underline{\otimes
}\varphi
\]

\end{enumerate}
\end{definition}

Now some comments on the above definition are in order.

\begin{remark}
\begin{enumerate}
\item How far the category $\mathcal{K}$\ should be exact is undoubtedly
disputable. Every geometer with the seven cardinal virtues agrees that the
category of smooth manifolds and smooth mappings is by no means exact enough.
However the requirement that $\mathcal{K}$\ should be a topos would presumably
be demanding too much so long as $\mathcal{K}$\ is expected to be naturally
realizable in our real world. Synthetic differential geometers have
constructed their well-adapted models, which are toposes, in their favotite
imaginary world. Our requirement in this paper that $\mathcal{K}$\ should be
left exact and cartesian closed is barely minimal without doubt. This point
will be discussed further in subsequent papers.

\item The functors $\mathbf{T}^{W}$'s stand for so-called Weil functors.

\item The conditions 4 and 6 in the above definition correspond in a sense to
(abstrct) Taylor expansion theorem in calculus or to what is dubbed the
(generalized) Kock-Lawvere axiom in synthetic differential geometry.

\item The formula (\ref{3.4}) has been inspired by Proposition in 35.18 of
\cite{kol}.

\item What is to be called the \underline{integration axiom} should
undoubtedly be considered. This point will be discussed in subsequent papers.
\end{enumerate}
\end{remark}

\begin{notation}
The natural transformation \frame{$\alpha_{\underline{\tau}_{W}}%
:\mathbf{T}^{W}\overset{\cdot}{\rightarrow}\mathrm{id}_{\mathcal{K}}$} is
denoted by $\tau_{W}$.
\end{notation}

\begin{notation}
The natural transformation \fbox{$\alpha_{\underline{\iota}_{W}}%
:\mathrm{id}_{\mathcal{K}}\overset{\cdot}{\rightarrow}\mathbf{T}^{W}$} is
denoted by $\iota_{W}$.
\end{notation}

It is easy to see (cf. Chapter II, \S 3, Proposition 1 of \cite{ma}) that

\begin{proposition}
\label{t3.1}Given a DG-category $\left(  \mathcal{K},\mathbb{R},\mathbf{T}%
,\alpha\right)  $, the pair $\left(  \mathbf{T},\alpha\right)  $ defines a
bifunctor $\otimes_{\mathbf{T},\alpha}:\mathcal{K}\times\mathbf{Weil}%
_{k}\mathbf{\rightarrow}\mathcal{K}$ in the sense that we have
\[
X\otimes_{\mathbf{T},\alpha}W=\mathbf{T}^{W}X
\]
for any object $X$\ in the category $\mathcal{K}$\ and any Weil $k$-algebra
$W$, while we have
\begin{align*}
& f\otimes_{\mathbf{T},\alpha}\varphi\\
& =\alpha_{\varphi,Y}\circ\mathbf{T}^{W_{1}}f\\
& =\mathbf{T}^{W_{2}}f\circ\alpha_{\varphi,X}%
\end{align*}
for any morphism $f:X\rightarrow Y$ in the category $\mathcal{K}$\ and any
morphism $\varphi:W_{1}\rightarrow W_{2}$ in the category $\mathbf{Weil}_{k} $.
\end{proposition}

\begin{notation}
We will often write $X\otimes W$ in place of $X\otimes_{\mathbf{T},\alpha}W $
unless any confusion may occur.
\end{notation}

We shall fix a DG-category $\left(  \mathcal{K},\mathbb{R},\mathbf{T}%
,\alpha\right)  $ throughout the rest of the paper.

\begin{definition}
An object $X$ in the category $\mathcal{K}$ is called \underline{\textit{Weil
exponentiable}}\textit{\ }if
\begin{equation}
(X\otimes(W_{1}\otimes_{k}W_{2}))^{Y}=(X\otimes W_{1})^{Y}\otimes
W_{2}\label{3.1}%
\end{equation}
holds naturally for any object $Y$ in the category $\mathcal{K}$ and any Weil
$k$-algebras $W_{1}$ and $W_{2}$.
\end{definition}

\begin{remark}
If $Y=1$, then (\ref{3.1}) degenerates into
\begin{equation}
X\otimes(W_{1}\otimes_{k}W_{2})=(X\otimes W_{1})\otimes W_{2}\label{3.2}%
\end{equation}
If $W_{1}=k$, then (\ref{3.1}) degenerates into
\begin{equation}
(X\otimes W_{2})^{Y}=X^{Y}\otimes W_{2}\label{3.3}%
\end{equation}

\end{remark}

\begin{proposition}
\label{t3.2}If $X$ is a Weil exponentiable object in the category
$\mathcal{K}$, then so is $X\otimes W$ for any Weil $k$-algebra $W$.
\end{proposition}

\begin{proof}
For any object $Y$ in the category $\mathcal{K}$ and any Weil $k$-algebras
$W_{1}$ and $W_{2}$, we have
\begin{align*}
&  ((X\otimes W)\otimes(W_{1}\otimes_{k}W_{2}))^{Y}\\
&  =(X\otimes((W\otimes_{k}W_{1})\otimes_{k}W_{2}))^{Y}\\
&  =(X\otimes(W\otimes_{k}W_{1}))^{Y}\otimes W_{2}\\
&  =((X\otimes W)\otimes W_{1})^{Y}\otimes W_{2}%
\end{align*}
so that we have the desired result.
\end{proof}

\begin{proposition}
\label{t3.3}If $\mathcal{F}:\Lambda\rightarrow\mathcal{K}$ is a finite diagram
in the category $\mathcal{K}$\ such that $\mathcal{F}_{\lambda}$ is Weil
exponentiable for any $\lambda\in\Lambda$, then $\mathrm{Lim}_{\lambda
\in\Lambda}\mathrm{\ }\mathcal{F}_{\lambda}$ is Weil exponentiable.
\end{proposition}

\begin{proof}
Since functors $\mathbf{T}^{W}:\mathcal{K}\rightarrow\mathcal{K}$ ($\forall
W\in\mathbf{Weil}_{k}$) and the exponentiation by $Y$ are left exact functors,
we have
\begin{align*}
& (\left(  \mathrm{Lim}_{\lambda\in\Lambda}\mathrm{\ }\mathcal{F}_{\lambda
}\right)  \otimes(W_{1}\otimes_{k}W_{2}))^{Y}\\
& =\left(  \mathrm{Lim}_{\lambda\in\Lambda}\mathrm{\ }\left(  \mathcal{F}%
_{\lambda}\otimes(W_{1}\otimes_{k}W_{2})\right)  \right)  ^{Y}\\
& =\mathrm{Lim}_{\lambda\in\Lambda}\mathrm{\ }\left(  \mathcal{F}_{\lambda
}\otimes(W_{1}\otimes_{k}W_{2})\right)  ^{Y}\\
& =\mathrm{Lim}_{\lambda\in\Lambda}\mathrm{\ }\left(  (\mathcal{F}_{\lambda
}\otimes W_{1})^{Y}\otimes W_{2}\right) \\
& =\left(  \mathrm{Lim}_{\lambda\in\Lambda}\mathrm{\ }(\mathcal{F}_{\lambda
}\otimes W_{1})^{Y}\right)  \otimes W_{2}\\
& =\left(  \mathrm{Lim}_{\lambda\in\Lambda}\mathrm{\ }(\mathcal{F}_{\lambda
}\otimes W_{1})\right)  ^{Y}\otimes W_{2}\\
& =\mathrm{\ }(\left(  \mathrm{Lim}_{\lambda\in\Lambda}\mathcal{F}_{\lambda
}\right)  \otimes W_{1})^{Y}\otimes W_{2}%
\end{align*}
so that we have the desired result.
\end{proof}

\begin{proposition}
\label{t3.4}If $X$ is a Weil exponentiable object in the category
$\mathcal{K}$, then so is $X^{Y}$ for any object $Y$ in the category
$\mathcal{K}$.
\end{proposition}

\begin{proof}
For any object $Z$ in category $\mathcal{K}$ and any Weil $k$-algebras $W_{1}
$ and $W_{2}$, we have
\begin{align*}
&  (X^{Y}\otimes(W_{1}\otimes_{k}W_{2}))^{Z}\\
&  =(X\otimes(W_{1}\otimes_{k}W_{2}))^{Y\times Z}\\
&  =(X\otimes W_{1})^{Y\times Z}\otimes W_{2}\\
&  =((X\otimes W_{1})^{Y})^{Z}\otimes W_{2}\\
&  =(X^{Y}\otimes W_{1})^{Z}\otimes W_{2}%
\end{align*}
so that we have the desired result.
\end{proof}

\begin{theorem}
\label{t3.5}The full subcategory $\mathcal{K}_{\mathbf{WE}}$\ of all Weil
exponentiable objects in the category $\mathcal{K}$ is a left exact and
cartesian closed category.
\end{theorem}

\begin{proof}
This follows simply from Propositions \ref{t3.3} and \ref{t3.4}.
\end{proof}

\begin{definition}
An object $X$ in the category $\mathcal{K}$ is called \underline
{\textit{microlinear}} providing that any finite limit diagram $\mathcal{D}$
in the category $\mathbf{Weil}_{k}$ yields a limit diagram $X\otimes
\mathcal{D}$ in $\mathcal{K}$, where $X\otimes\mathcal{D}$ is obtained from
$\mathcal{D}$ by putting $X\otimes$ to the left of every object and every
morphism in $\mathcal{D}$.
\end{definition}

\begin{proposition}
\label{t3.6}If an object $X$ in the category $\mathcal{K}$ is Weil
exponentiable and microlinear, then so is $X\otimes W$ for any Weil
$k$-algebra $W$.
\end{proposition}

\begin{proof}
Given a finite limit diagram $\mathcal{D}$ in the category $\mathbf{Weil}_{k}
$, we have
\[
\left(  X\otimes W\right)  \otimes\mathcal{D}=X\otimes(W\otimes_{k}%
\mathcal{D})
\]
by (\ref{3.2}). Since the functor $W\otimes_{k}\cdot:\mathbf{Weil}%
_{k}\mathbf{\rightarrow Weil}_{k}$ preserves finite limits, we have the
desired result.
\end{proof}

\begin{proposition}
\label{t3.7}If $\mathcal{F}:\Lambda\rightarrow\mathcal{K}$ is a finite diagram
in the category $\mathcal{K}$\ such that $\mathcal{F}_{\lambda}$ is
microlinear object in $\mathcal{K}$ for any $\lambda\in\Lambda$, then its
limit $\mathrm{Lim}_{\lambda\in\Lambda}\mathrm{\ }\mathcal{F}_{\lambda}$ is
also a microlinear object in $\mathcal{K}$.
\end{proposition}

\begin{proof}
Given a finite diagram $\mathcal{D}:\Gamma\rightarrow\mathbf{Weil}_{k}$ in the
category $\mathbf{Weil}_{k}$, we have
\begin{align*}
& \mathrm{Lim}_{\gamma\in\Gamma}\mathrm{\ }\left(  \left(  \mathrm{Lim}%
_{\lambda\in\Lambda}\mathrm{\ }\mathcal{F}_{\lambda}\right)  \otimes
\mathcal{D}_{\gamma}\right) \\
& =\mathrm{Lim}_{\gamma\in\Gamma}\mathrm{\ }\left(  \mathrm{Lim}_{\lambda
\in\Lambda}\mathrm{\ }\left(  \mathcal{F}_{\lambda}\otimes\mathcal{D}_{\gamma
}\right)  \right) \\
& =\mathrm{Lim}_{\lambda\in\Lambda}\mathrm{\ }\left(  \mathrm{Lim}_{\gamma
\in\Gamma}\mathrm{\ }\left(  \mathcal{F}_{\lambda}\otimes\mathcal{D}_{\gamma
}\right)  \right) \\
& \text{[since double limits commute]}\\
& =\mathrm{Lim}_{\lambda\in\Lambda}\mathrm{\ }\left(  \mathcal{F}_{\lambda
}\otimes\left(  \mathrm{Lim}_{\gamma\in\Gamma}\mathrm{\ }\mathcal{D}_{\gamma
}\right)  \right) \\
& \text{[since }\mathcal{F}_{\lambda}\text{\ is microlinear]}\\
& =\left(  \mathrm{Lim}_{\lambda\in\Lambda}\mathrm{\ }\mathcal{F}_{\lambda
}\right)  \otimes\left(  \mathrm{Lim}_{\gamma\in\Gamma}\mathrm{\ }%
\mathcal{D}_{\gamma}\right)
\end{align*}
so that we have the desired result.
\end{proof}

\begin{proposition}
\label{t3.8}If $X$ is a Weil exponentiable and microlinear object in
$\mathcal{K}$, then so is $X^{Y}$ for any object $Y$ in $\mathcal{K}$.
\end{proposition}

\begin{proof}
$X^{Y}$ is Weil exponentiable by Proposition \ref{t3.4}. Given a finite
diagram $\mathcal{D}:\Gamma\rightarrow\mathbf{Weil}_{k}$, we have
\begin{align*}
& \mathrm{Lim}_{\gamma\in\Gamma}\mathrm{\ }\left(  X^{Y}\otimes\mathcal{D}%
_{\gamma}\right) \\
& =\mathrm{Lim}_{\gamma\in\Gamma}\mathrm{\ }\left(  X\otimes\mathcal{D}%
_{\gamma}\right)  ^{Y}\\
& \text{[by (\ref{3.3})]}\\
& =\left(  \mathrm{Lim}_{\gamma\in\Gamma}\mathrm{\ }\left(  X\otimes
\mathcal{D}_{\gamma}\right)  \right)  ^{Y}\\
& \text{[since the exponentiation by }Y\text{\ is a left exact functor]}\\
& =\left(  X\otimes\left(  \mathrm{Lim}_{\gamma\in\Gamma}\mathrm{\ }%
\mathcal{D}_{\gamma}\right)  \right)  ^{Y}\\
& \text{[since }X\text{\ is microlinear]}\\
& =X^{Y}\otimes\left(  \mathrm{Lim}_{\gamma\in\Gamma}\mathrm{\ }%
\mathcal{D}_{\gamma}\right) \\
& \text{[by (\ref{3.3})]}%
\end{align*}
so that $X^{Y}$ is microlinear.
\end{proof}

Now we recapitulate as follows.

\begin{theorem}
\label{t3.9}The full subcategory $\mathcal{K}_{\mathbf{WE},\mathbf{ML}}$ of
all Weil exponentiable and microlinear objects in the category $\mathcal{K}$
is a left exact and cartesian closed category.
\end{theorem}

\section{\label{s4}Fibered Microlinear Objects}

Now we are going to talk about fibered manifolds in our context. First of all,
we note that

\begin{proposition}
\label{t4.1}The quadruple $\left(  \overrightarrow{\mathcal{K}}%
,\overrightarrow{\mathbb{R}},\overrightarrow{\mathbf{T}},\overrightarrow
{\alpha}\right)  $\ is a DG-category, where

\begin{enumerate}
\item $\overrightarrow{\mathcal{K}}$ is the arrow category of $\mathcal{K}$.

\item $\overrightarrow{\mathbb{R}}$ stands for
\[%
\begin{array}
[c]{c}%
\mathbb{R}\\
\downarrow\\
1
\end{array}
\]

\item Given a Weil $k$-algebra $W$, $\overrightarrow{\mathbf{T}}^{W}\left(
\begin{array}
[c]{c}%
E\\%
\begin{array}
[c]{cc}%
\pi & \downarrow
\end{array}
\\
M
\end{array}
\right)  $ is
\[%
\begin{array}
[c]{c}%
\mathbf{T}^{W}\left(  E\right) \\%
\begin{array}
[c]{cc}%
\mathbf{T}^{W}\left(  \pi\right)  & \downarrow
\end{array}
\\
\mathbf{T}^{W}\left(  M\right)
\end{array}
\]
while $\overrightarrow{\mathbf{T}}^{W}\left(
\begin{array}
[c]{ccc}%
E_{1} & \underrightarrow{\quad f\quad} & E_{2}\\%
\begin{array}
[c]{cc}%
\pi_{1} & \downarrow
\end{array}
&  &
\begin{array}
[c]{cc}%
\downarrow & \pi_{2}%
\end{array}
\\
M_{1} & \underrightarrow{\quad\overline{f}\quad} & M_{2}%
\end{array}
\right)  $ is
\[%
\begin{array}
[c]{ccc}%
\mathbf{T}^{W}\left(  E_{1}\right)  & \underrightarrow{\quad\mathbf{T}%
^{W}\left(  f\right)  \quad} & \mathbf{T}^{W}\left(  E_{2}\right) \\%
\begin{array}
[c]{cc}%
\mathbf{T}^{W}\left(  \pi_{1}\right)  & \downarrow
\end{array}
&  &
\begin{array}
[c]{cc}%
\downarrow & \mathbf{T}^{W}\left(  \pi_{2}\right)
\end{array}
\\
\mathbf{T}^{W}\left(  M_{1}\right)  & \underrightarrow{\quad\mathbf{T}%
^{W}\left(  \overline{f}\right)  \quad} & \mathbf{T}^{W}\left(  M_{2}\right)
\end{array}
\]

\item Given a morphism $\varphi:W_{1}\rightarrow W_{2}$ in the category
$\mathbf{Weil}_{k}$, $\overrightarrow{\alpha}_{\varphi}\left(
\begin{array}
[c]{c}%
E\\%
\begin{array}
[c]{cc}%
\pi & \downarrow
\end{array}
\\
M
\end{array}
\right)  $ is
\[%
\begin{array}
[c]{ccc}%
\mathbf{T}^{W_{1}}\left(  E\right)  & \underrightarrow{\alpha_{\varphi}\left(
E\right)  } & \mathbf{T}^{W_{2}}\left(  E\right) \\%
\begin{array}
[c]{cc}%
\mathbf{T}^{W_{1}}\left(  \pi\right)  & \downarrow
\end{array}
&  &
\begin{array}
[c]{cc}%
\downarrow & \mathbf{T}^{W_{2}}\left(  \pi\right)
\end{array}
\\
\mathbf{T}^{W_{1}}\left(  M\right)  & \overrightarrow{\alpha_{\varphi}\left(
M\right)  } & \mathbf{T}^{W_{2}}\left(  M\right)
\end{array}
\]

\end{enumerate}
\end{proposition}

\begin{proof}
That the category $\overrightarrow{\mathcal{K}}$\ is left exact and that the
functor $\overrightarrow{\mathbf{T}}^{W}:\overrightarrow{\mathcal{K}%
}\rightarrow\overrightarrow{\mathcal{K}}$\ is left exact follow at the same
time from Theorem 7.5.2 in \cite{sch}. That the category $\overrightarrow
{\mathcal{K}}$\ is cartesian closed follows from Exercise 1.3.7 in \cite{j1}.
The other conditions for $\left(  \overrightarrow{\mathcal{K}},\overrightarrow
{\mathbb{R}},\overrightarrow{\mathbf{T}},\overrightarrow{\alpha}\right)  $\ to
be a DG-category are easy to verify.
\end{proof}

\begin{corollary}
\label{t4.2}$%
\begin{array}
[c]{c}%
E\\%
\begin{array}
[c]{cc}%
\pi & \downarrow
\end{array}
\\
M
\end{array}
\in\overrightarrow{\mathcal{K}}$ is microlinear with respect to the
DG-category $\left(  \overrightarrow{\mathcal{K}},\overrightarrow{\mathbb{R}%
},\overrightarrow{\mathbf{T}},\overrightarrow{\alpha}\right)  $ iff both
$E$\ and $M$\ are microlinear with respect to the DG-category $\left(
\mathcal{K},\mathbb{R},\mathbf{T},\alpha\right)  $.
\end{corollary}

\begin{remark}
\label{r4.1}Given two objects $\pi:E\rightarrow M$ and $\theta:F\rightarrow N$
in the category $\overrightarrow{\mathcal{K}}$, their exponential $\pi
^{\theta}:\left(  E^{F}\right)  _{\mathbf{P}}\rightarrow M^{N}$ is determined
by the pullback diagram
\[%
\begin{array}
[c]{ccc}%
\left(  E^{F}\right)  _{\mathbf{P}} & \rightarrow & E^{F}\\
\downarrow &  & \downarrow\\
M^{N} & \rightarrow & M^{F}%
\end{array}
\]
where $\mathbf{P}$\ stands for ''Projectable (into $M^{N}$)''.
\end{remark}

\begin{proposition}
\label{t4.8}Given a morphism $\pi:E\rightarrow M$ in the category
$\mathcal{K}$, if both $E$ and $M$\ are Weil exponentiable as objects in the
category $\mathcal{K}$, then $\pi:E\rightarrow M$\ is Weil exponentiable as an
object in the category $\overrightarrow{\mathcal{K}}$.
\end{proposition}

\begin{proof}
As we have noted, $\left(  \left(  E\otimes W_{1}\right)  ^{F}\right)
_{\mathbf{P}}$ is obtained as the pullback of the diagram
\begin{equation}%
\begin{array}
[c]{ccc}%
\left(  \left(  E\otimes W_{1}\right)  ^{F}\right)  _{\mathbf{P}} &
\rightarrow & \left(  E\otimes W_{1}\right)  ^{F}\\
\downarrow &  & \downarrow\\
\left(  M\otimes W_{1}\right)  ^{N} & \rightarrow & \left(  M\otimes
W_{1}\right)  ^{F}%
\end{array}
\label{4.1}%
\end{equation}
Since the functor $\cdot\otimes W_{2}:\mathcal{K}\rightarrow\mathcal{K}$ is
left exact, the diagram obtained from (\ref{4.1}) by the application of the
functor
\begin{equation}%
\begin{array}
[c]{ccc}%
\left(  \left(  E\otimes W_{1}\right)  ^{F}\right)  _{\mathbf{P}}\otimes W_{2}
& \rightarrow &
\begin{array}
[c]{c}%
\left(  E\otimes W_{1}\right)  ^{F}\otimes W_{2}\\
=\left(  E\otimes\left(  W_{1}\otimes_{k}W_{2}\right)  \right)  ^{F}%
\end{array}
\\
\downarrow &  & \downarrow\\%
\begin{array}
[c]{c}%
\left(  M\otimes W_{1}\right)  ^{N}\otimes W_{2}\\
=\left(  M\otimes\left(  W_{1}\otimes_{k}W_{2}\right)  \right)  ^{N}%
\end{array}
& \rightarrow &
\begin{array}
[c]{c}%
\left(  M\otimes W_{1}\right)  ^{F}\otimes W_{2}\\
=\left(  M\otimes\left(  W_{1}\otimes_{k}W_{2}\right)  \right)  ^{F}%
\end{array}
\end{array}
\label{4.2}%
\end{equation}
is a pullback diagram. However the diagram
\begin{equation}%
\begin{array}
[c]{ccc}%
\left(  \left(  E\otimes\left(  W_{1}\otimes_{k}W_{2}\right)  \right)
^{F}\right)  _{\mathbf{P}} & \rightarrow & \left(  E\otimes\left(
W_{1}\otimes_{k}W_{2}\right)  \right)  ^{F}\\
\downarrow &  & \downarrow\\
\left(  M\otimes\left(  W_{1}\otimes_{k}W_{2}\right)  \right)  ^{N} &
\rightarrow & \left(  M\otimes\left(  W_{1}\otimes_{k}W_{2}\right)  \right)
^{F}%
\end{array}
\label{4.3}%
\end{equation}
is a pullback diagram. Therefore we have
\[
\left(  \left(  E\otimes\left(  W_{1}\otimes_{k}W_{2}\right)  \right)
^{F}\right)  _{\mathbf{P}}=\left(  \left(  E\otimes W_{1}\right)  ^{F}\right)
_{\mathbf{P}}\otimes W_{2}%
\]
which is the desired result.
\end{proof}

\begin{definition}
By a \underline{fibered microlinear object in $\mathcal{K}$} we mean simply an
object $\pi:E\rightarrow M$ in the category $\overrightarrow{\mathcal{K}} $
which is Weil exponentiable and microlinear with respet to the DG-structure
$\left(  \overrightarrow{\mathcal{K}},\overrightarrow{\mathbb{R}%
},\overrightarrow{\mathbf{T}},\overrightarrow{\alpha}\right)  $.
\end{definition}

\begin{notation}
The full subcategory of $\overrightarrow{\mathcal{K}}$ consisting of all fiber
bundles in $\mathcal{K}\ $is denoted by $\mathcal{K}_{\mathbf{Fib}}$.
\end{notation}

\begin{theorem}
\label{t4.9}The category $\mathcal{K}_{\mathbf{Fib}}$\ is left exact and
cartesian closed.
\end{theorem}

\begin{proof}
This follows directly from Theorem \ref{t3.9}.
\end{proof}

\section{\label{s5}Vertical Constructions}

Now we are going to discuss vertical bundles in our context.

\begin{definition}
Given a morphism $\pi:E\rightarrow M$ in the category $\mathcal{K}$ and a Weil
$k$-algebra $W$, the \underline{vertical bundle} $\tau_{W}^{\mathbf{V}}\left(
\pi\right)  :\mathbf{V}^{W}\left(  \pi\right)  \rightarrow E$\ of $\pi$\ with
respect to $W$\ is defined to be
\[
\tau_{W}^{\mathbf{V}}\left(  \pi\right)  =\tau_{W}\left(  E\right)
\circ\widetilde{\tau}_{W}^{\mathbf{V}}\left(  \pi\right)
\]
where $\widetilde{\tau}_{W}^{\mathbf{V}}\left(  \pi\right)  :\mathbf{V}%
^{W}\left(  \pi\right)  \rightarrow E\otimes W$ is obtained as the equalizer
of
\[
E\otimes W\underrightarrow{\quad\pi\otimes\mathrm{id}_{W}\quad}M\otimes W
\]
and
\[
E\otimes W\underrightarrow{\quad\tau_{W,E}\quad}E\underrightarrow{\quad
\pi\quad}M\underrightarrow{\quad\iota_{W,M}\quad}M\otimes W
\]

\end{definition}

\begin{notation}
We will often write $E\otimes^{\perp}W$ for $\mathbf{V}^{W}\left(  \pi\right)
$.
\end{notation}

\begin{lemma}
\label{t4.3}Given a diagram in a left exact category $\mathcal{J}$%
\[%
\begin{array}
[c]{ccc}%
Z_{1} &  & Z_{2}\\
\downarrow &  & \downarrow\\
X_{1} & \underrightarrow{\quad f\quad} & X_{2}\\%
\begin{array}
[c]{ccc}%
g_{1} & \downdownarrows & h_{1}%
\end{array}
&  &
\begin{array}
[c]{ccc}%
g_{2} & \downdownarrows & h_{2}%
\end{array}
\\
Y_{1} & \underrightarrow{\quad\overline{f}\quad} & Y_{2}%
\end{array}
\]
if both of the two diagrams
\[%
\begin{array}
[c]{c}%
Z_{1}\\
\downarrow\\
X_{1}\\%
\begin{array}
[c]{ccc}%
g_{1} & \downdownarrows & h_{1}%
\end{array}
\\
Y_{1}%
\end{array}
\]
\[%
\begin{array}
[c]{c}%
Z_{2}\\
\downarrow\\
X_{2}\\%
\begin{array}
[c]{ccc}%
g_{2} & \downdownarrows & h_{2}%
\end{array}
\\
Y_{2}%
\end{array}
\]
are equalizers, and if both of the diagrams
\[%
\begin{array}
[c]{ccc}%
X_{1} & \underrightarrow{\quad f\quad} & X_{2}\\%
\begin{array}
[c]{cc}%
g_{1} & \downarrow
\end{array}
&  &
\begin{array}
[c]{cc}%
\downarrow & g_{2}%
\end{array}
\\
Y_{1} & \underrightarrow{\quad\overline{f}\quad} & Y_{2}%
\end{array}
\]
\[%
\begin{array}
[c]{ccc}%
X_{1} & \underrightarrow{\quad f\quad} & X_{2}\\%
\begin{array}
[c]{cc}%
h_{1} & \downarrow
\end{array}
&  &
\begin{array}
[c]{cc}%
\downarrow & h_{2}%
\end{array}
\\
Y_{1} & \underrightarrow{\quad\overline{f}\quad} & Y_{2}%
\end{array}
\]
commute, then there is a unique morphism $Z_{1}\dashrightarrow Z_{2}$\ making
the diagram
\[%
\begin{array}
[c]{ccc}%
Z_{1} & \dashrightarrow & Z_{2}\\
\downarrow &  & \downarrow\\
X_{1} & \underrightarrow{\quad f\quad} & X_{2}%
\end{array}
\]
commutative.
\end{lemma}

\begin{proof}
By the familiar token of what is dubbed arrow chasing.
\end{proof}

\begin{corollary}
\label{t4.4}Given a Weil $k$-algebra $W$, our previous mapping $\mathbf{V}%
^{W}$ assigning the object $\mathbf{V}^{W}\left(  \pi\right)  $ in the
category $\mathcal{K}$\ to each object
\[%
\begin{array}
[c]{c}%
E\\%
\begin{array}
[c]{cc}%
\pi & \downarrow
\end{array}
\\
M
\end{array}
\]
in the category $\overrightarrow{\mathcal{K}}$ can naturally be extended to a
functor $\mathbf{V}^{W}:\overrightarrow{\mathcal{K}}\mathcal{\rightarrow K} $
in the sense that the diagram
\[%
\begin{array}
[c]{ccc}%
\mathbf{V}^{W}\left(  \pi_{1}\right)  & \underrightarrow{\mathbf{V}^{W}\left(
\left(  f,\overline{f}\right)  \right)  } & \mathbf{V}^{W}\left(  \pi
_{2}\right) \\%
\begin{array}
[c]{cc}%
\tau_{W}^{\mathbf{V}}\left(  \pi_{1}\right)  & \downarrow
\end{array}
&  &
\begin{array}
[c]{cc}%
\downarrow & \tau_{W}^{\mathbf{V}}\left(  \pi_{2}\right)
\end{array}
\\
E_{1} & \underrightarrow{\quad f\quad} & E_{2}%
\end{array}
\]
commutes for any morphism
\[%
\begin{array}
[c]{ccc}%
E_{1} & \underrightarrow{\quad f\quad} & E_{2}\\%
\begin{array}
[c]{cc}%
\pi_{1} & \downarrow
\end{array}
&  &
\begin{array}
[c]{cc}%
\downarrow & \pi_{2}%
\end{array}
\\
M_{1} & \underrightarrow{\quad\overline{f}\quad} & M_{2}%
\end{array}
\]
in the category $\overrightarrow{\mathcal{K}}$.
\end{corollary}

\begin{proof}
It suffices to note that
\[%
\begin{array}
[c]{ccc}%
E_{1}\otimes W & \underrightarrow{\quad f\otimes\mathrm{id}_{W}\quad} &
E_{2}\otimes W\\%
\begin{array}
[c]{cc}%
\tau_{W,E_{1}} & \downarrow
\end{array}
&  &
\begin{array}
[c]{cc}%
\downarrow & \tau_{W,E_{2}}%
\end{array}
\\
E_{1} & \underrightarrow{\quad f\quad} & E_{2}\\%
\begin{array}
[c]{cc}%
\pi_{1} & \downarrow
\end{array}
&  &
\begin{array}
[c]{cc}%
\downarrow & \pi_{2}%
\end{array}
\\
M_{1} & \underrightarrow{\quad\overline{f}\quad} & M_{2}\\%
\begin{array}
[c]{cc}%
\iota_{W,M_{1}} & \downarrow
\end{array}
&  &
\begin{array}
[c]{cc}%
\downarrow & \iota_{W,M_{2}}%
\end{array}
\\
M_{1}\otimes W & \underrightarrow{\quad\overline{f}\otimes\mathrm{id}_{W}%
\quad} & M_{2}\otimes W
\end{array}
\]
In particular, the outer rectangle is commutative, so that the desired result
follows directly from the lemma.
\end{proof}

\begin{lemma}
\label{t4.5}Given three finite diagrams $\mathcal{F},\mathcal{G}%
,\mathcal{H}:\Lambda\rightarrow\mathcal{J}$ in a left exact category
$\mathcal{J}$ with the same underlying category $\Lambda$, two natural
transformations $\mu,\nu:\mathcal{G}\overset{\cdot}{\rightarrow}\mathcal{H}$
and a natural transformation $\theta:\mathcal{F}\overset{\cdot}{\rightarrow
}\mathcal{G}$, if the diagrams $\mathcal{G}$ and $\mathcal{H}$ are limit
diagrams, and if the diagram
\[%
\begin{array}
[c]{c}%
\mathcal{F}_{\lambda}\\%
\begin{array}
[c]{cc}%
\theta_{\lambda} & \downarrow
\end{array}
\\
\mathcal{G}_{\lambda}\\%
\begin{array}
[c]{ccc}%
\mu_{\lambda} & \downdownarrows & \nu_{\lambda}%
\end{array}
\\
\mathcal{H}_{\lambda}%
\end{array}
\]
is an equalizer for each $\lambda\in\Lambda$, then the diagram $\mathcal{F}%
$\ is a limit diagram.
\end{lemma}

\begin{proof}
By the familiar token of what is called arrow chasing.
\end{proof}

\begin{corollary}
\label{t4.6}Given a Weil $k$-algebra $W$, the functor $\mathbf{V}%
^{W}:\overrightarrow{\mathcal{K}}\mathcal{\rightarrow K}$ is left exact.
\end{corollary}

\begin{proof}
Given a finite limit diagram in the category $\overrightarrow{\mathcal{K}}$,
which decomposes into two limit diagrams $\mathcal{G}^{\prime},\mathcal{H}%
^{\prime}:\Lambda\rightarrow\mathcal{K}$ in the category $\mathcal{K}$\ and a
natural transformation $\mu^{\prime}:\mathcal{G}^{\prime}\overset{\cdot
}{\rightarrow}\mathcal{H}^{\prime}$, we have two limit diagrams
\[
\mathcal{G=G}^{\prime}\otimes W
\]
\[
\mathcal{H=H}^{\prime}\otimes W
\]
and two natural transformations
\[
\mu=\mu^{\prime}\otimes W:\mathcal{G}\overset{\cdot}{\rightarrow}\mathcal{H}%
\]
\[
\nu:\mathcal{G}\overset{\cdot}{\overrightarrow{\tau_{W}}}\mathcal{G}^{\prime
}\overset{\cdot}{\overrightarrow{\mu^{\prime}}}\mathcal{H}^{\prime}%
\overset{\cdot}{\overrightarrow{\iota_{W}}}\mathcal{H}%
\]
Therefore the desired result follows from the lemma.
\end{proof}

\begin{corollary}
\label{t4.7}(\underline{The Vertical Microlinearity Theorem}) Let
$\pi:E\rightarrow M$ be a morphism in the category $\mathcal{K}$ with $E$\ and
$M $\ being microlinear. If $\mathcal{D}$ is a finite limit diagram in the
category $\mathbf{Weil}_{k}$ , then $E\otimes^{\perp}\mathcal{D}$ is a limit
diagram in the category $\mathcal{K}$.
\end{corollary}

\begin{proof}
It suffices to note that both $E\otimes\mathcal{D}$\ standing for
$\mathcal{G}$\ in the above lemma and $M\otimes\mathcal{D}$\ standing for
$\mathcal{H} $\ in the above lemma are limit diagrams, because $E$\ and
$M$\ are assumed to be microlinear. Then the desired result follows directly
from the lemma.
\end{proof}

\end{document}